\documentclass[a4paper,11pt,reqno]{article}

\usepackage{color}
\definecolor{red}{rgb}{1,0,0}
\definecolor{green}{rgb}{0,1,0}
\definecolor{blue}{rgb}{0,0,1}
\definecolor{refkey}{gray}{.625}
\definecolor{labelkey}{gray}{.625}

\let\oldmarginpar\marginpar
\renewcommand\marginpar[1]{\-\oldmarginpar[\raggedleft\footnotesize #1]%
{\raggedright\footnotesize #1}}

\usepackage{verbatim}

\usepackage[T1]{fontenc}
\usepackage{lmodern}
\usepackage[latin1]{inputenc}

\usepackage{setspace}
\usepackage{indentfirst}

\usepackage{geometry}
\usepackage{hyperref}

\usepackage{amsrefs}

\usepackage[intlimits]{amsmath}
\usepackage{amssymb}
\usepackage{mathrsfs}
\usepackage{stmaryrd}
\usepackage{mathtools}

\usepackage{amscd}

\usepackage{graphicx}
\usepackage[all]{xy}
\def\dar[#1]{\ar@<2pt>[#1]\ar@<-2pt>[#1]}

\usepackage{amsthm}

\theoremstyle{plain}
\newtheorem{prop}{Proposition}[section]
\newtheorem{lem}[prop]{Lemma}

\newtheorem{cor}[prop]{Corollary}
\newtheorem{thm}[prop]{Theorem}

\newtheorem*{prop*}{Proposition}
\newtheorem*{lem*}{Lemma}
\newtheorem*{sublem*}{Sublemma}
\newtheorem*{cor*}{Corollaire}
\newtheorem*{thm*}{Theorem}
\newtheorem*{hypo*}{Hypothesis}
\newtheorem*{question*}{Question}
\newtheorem*{conjecture*}{Conjecture}
\newtheorem*{scholum*}{Scholum}
\newtheorem{defn}[prop]{Definition}
\newtheorem*{defn*}{Definition}

\newtheoremstyle{slanted}
  {3pt}
  {3pt}
  {\slshape}
  {}
  {\bfseries}
  {.}
  {.5em}
  {}

\theoremstyle{slanted}
\newtheorem{ex}[prop]{Example}

\newtheorem*{ex*}{Example}
\newtheorem*{exs*}{Examples}
\newtheorem{rmk}[prop]{Remark}

\newtheorem*{rmk*}{Remark}
\newtheorem*{rmks*}{Remarks}

\newtheorem*{notation*}{Notation}

\theoremstyle{definition}

\newtheorem*{con*}{Construction}
\newtheorem*{note*}{Note}

\theoremstyle{remark}

\newtheorem*{warning*}{Warning}
\newtheorem*{shortnote*}{Note}
\newtheorem*{claim*}{Claim}
\newtheorem*{axiom*}{Axiom}

\newcommand{\BT}[1]{\begin{thm}\label{T:#1}}
\newcommand{\ET}{\end{thm}}

\newcommand{\BL}[1]{\begin{lem}\label{L:#1}}
\newcommand{\EL}{\end{lem}}

\newcommand{\BC}[1]{\begin{cor}\label{C:#1}}
\newcommand{\EC}{\end{cor}}

\newcommand{\BP}[1]{\begin{prop}\label{P:#1}}
\newcommand{\EP}{\end{prop}}

\newcommand{\BD}[1]{\begin{defn}\label{D:#1}}
\newcommand{\ED}{\end{defn}}

\newcommand{\BE}[1]{\begin{ex}\label{E:#1}}
\newcommand{\EE}{\end{ex}}

\newcommand{\BR}[1]{\begin{rmk}\label{R:#1}}
\newcommand{\ER}{\end{rmk}}

\newcommand{\BM}[1]{\begin{equation}\label{M:#1}}
\newcommand{\EM}{\end{equation}}

\newcommand{\reals}{\mathbb{R}}

\newcommand{\thalf}{\tfrac{1}{2}}

\newcommand{\rond}{\circ}

\newcommand{\cinf}[1]{C^{\infty}(#1)}
\newcommand{\sections}[1]{ {\Gamma} (#1)}
\newcommand{\XX}{\mathfrak{X}} 
\newcommand{\OO}{\Omega} 

\newcommand{\isomorphism}{\cong}

\newcommand{\diese}{^{\sharp}}

\newcommand{\graded}{^{\scriptscriptstyle\bullet}}

\newcommand{\mfg}{\mathfrak{g}} 

\newcommand{\DD}{\mathcal{D}}

\newcommand{\ba}[2]{[#1,#2]}
\newcommand{\bas}[2]{[#1,#2]_*}

\newcommand{\ld}[1]{\mathcal{L}_{#1}}
\newcommand{\ii}[1]{\iota_{#1}}

\newcommand{\ip}[2]{\langle #1 , #2 \rangle}
\newcommand{\cb}[2]{\llbracket #1 , #2 \rrbracket}
\newcommand{\lb}[2]{[ #1 , #2 ]}
\newcommand{\db}[2]{#1 \circ #2}

\newcommand{\anchor}{\rho}
\newcommand{\anchors}{\anchor_*}
\newcommand{\dee}{d}
\newcommand{\dees}{\dee_*}

\newcommand{\lon }{\,\rightarrow\,}

\newcommand{\As}{A^*}

\newcommand{\dA}{d_A }

\newcommand{\LieDer}{\mathcal{L} }
\newcommand{\pairing}[2]{\left\langle #1  ~|~  #2 \right\rangle}

\newcommand{\inserts}{\iota}

\newcommand{\bookpairing}[1]{            {\Bigl\lfloor#1\Bigr\rfloor}}
\newcommand{\set}[1]{\left\{#1\right\}}

\newcommand{\defbe}{:=}

\newcommand{\Lambdabracket}[1]{\left [ #1\right ]_{\Lambda}}

\newcommand{\Lambdasharp}{\Lambda^\sharp}

\newcommand{\Ker}{\mathrm{Ker}}

\newcommand{\thetaalgebra}{\mathfrak{\theta}}

\newcommand{\thetaalgebrastar}{\mathfrak{\theta}^*}
\newcommand{\galgebra}{\mathfrak{g}}
\newcommand{\galgebrastar}{\mathfrak{g}^*}

\newcommand{\coadjoint}[1]{\mathrm{ad}^*_{#1}}

\newcommand{\phiUprotate }{\phi^{\scriptscriptstyle T}}

\newcommand{\Rmatrix}{\mathbf{r}}

\newcommand{\crossedmoduletriple}[3]{(#1\stackrel{#2}{\rightharpoonup}#3)}

\newcommand{\Ps}{P^*}
\newcommand{\Qs}{Q^*}

\newcommand{\moduleaction}{ \triangleright}

\newcommand{\symmetricproduct}{{\scriptstyle \odot}\,}

\newcommand{\tobefilledin}{\,\stackrel{\centerdot}{}\,}

\newcommand{\algebracrossedmodulemap}{\phi}

\hypersetup{pdfpagemode=UseOutlines,colorlinks=false,pdfpagelayout=SinglePage,pdfstartview=FitH,bookmarksopen=true}

\begin{document}

\title{On Lie Bialgebroid Crossed Modules
\footnote{Research is partially supported by NSFC grants  11471179, 11301317, 11971282.}}

\author{\textsc{Honglei Lang} \\
{\small Department of Applied Mathematics}\\{\small China Agricultural University, Beijing 100083, China}\\
{\small
{\href{hllang@cau.edu.cn}{\texttt{hllang@cau.edu.cn
}}}}
\and\textsc{Yu Qiao} \\
{\small School of Mathematics and Information Science}\\
{\small Shaanxi Normal University}\\{\small Xi'an, China}\\
{\small
{\href{yqiao@snnu.edu.cn}{\texttt{yqiao@snnu.edu.cn
}}}}
\and
\textsc{Yanbin Yin\footnote{Corresponding author.}} \\
{\small School of Mathematics and Statistics}\\
{\small Henan University}\\{\small Kaifeng, China}\\
{\small
{\href{yyb@henu.edu.cn}{\texttt{yyb@henu.edu.cn
}}}}
 }

\date{}

\maketitle
\begin{abstract}
We study Lie bialgebroid crossed modules which are
 pairs  of Lie algebroid crossed modules in duality
 that canonically give rise to Lie bialgebroids.
 A one-one correspondence between such Lie bialgebroid crossed modules and co-quadratic Manin triples  $(K,P,Q)$ is established, where $K$ is a co-quadratic Lie algebroid and $(P,Q)$ is a pair of transverse Dirac structures in $K$.
\end{abstract}

 \tableofcontents

\section{Introduction}

The theory of Lie bialgebras and Poisson
Lie groups, due for the most part to V. G. Drinfel{\cprime}d and M. A. Semenov-Tyan-Shanski{\u\i}, dates back to the early 80's \cite{MR688240,MR725413}. A Lie bialgebra is a pair of Lie algebras $(\galgebra,\galgebrastar)$ in duality satisfying some compatible conditions. An interesting fact is that, for a Lie bialgebra $(\galgebra,\galgebrastar)$, there {is} a quadratic Lie algebra $ D=\galgebra\bowtie\galgebrastar$, namely the double Lie algebra, and both $\galgebra$ and $\galgebrastar$ serve as maximally coisotropic subalgebras in $D$.  In fact, $(D,\galgebra,\galgebrastar)$ is an example of Manin triples \cite{MR934283}. An important result is that Lie bialgebras are in one-one correspondence with Manin triples.

Lie bialgebroids, first introduced by  Mackenzie and Xu in \cite{MR1262213}, are a natural generalization of Lie bialgebras.
A significant difference is that, the double of a Lie bialgebroid $(A,\As)$ is not a quadratic Lie algebroid. Instead, there  {is} a Courant algebroid structure underlying $E=A\oplus \As$, and both $A$ and $\As$ are maximally isotropic and integrable (so called Dirac structures) in $E$ \cite{MR1472888}. Again, $(E,A,\As)$ forms a Manin triple, and the one-one correspondence between Lie biaglebroids and Manin triples still holds.

 {
The concept of Lie group crossed modules was proposed by Whitehead \cite{Whitehead1, Whitehead2} in 1940s. Since then, it has been widely used
in areas such as homotopy theory \cite{BHS}, group representations \cite{Brown1}, algebraic $K$-theory \cite{Loday}, and homological algebras \cite{Hue, Lue}. The category of Lie group crossed modules is equivalent to the category of group-groupoids \cite{BS}, and to the category of $cat^1$-groups \cite{Loday1}. Lie group crossed modules can be also viewed as 2-dimensional groups \cite{Brown}.  Lie algebra crossed modules also appeared in Gerstenhaber's work \cite{Gers}.
}

In this note, we will introduce a notion called ``Lie bialgebroid crossed modules'' and show its  relation with many other objects.   Let us first  review some basic notions that are needed in this paper.

A Lie algebra crossed module is a pair of Lie algebras $(\thetaalgebra,\galgebra)$, which are related by a Lie algebra morphism $\phi:\thetaalgebra\longrightarrow \galgebra$, and satisfies some compatibility conditions. A well-known result is that the group of equivalence classes of crossed modules $\phi:\thetaalgebra\longrightarrow \galgebra$ describes the third cohomology group of the Lie algebra $\galgebra/\phi(\thetaalgebra)$ \cite{MR0207793}.  Moreover, $(\thetaalgebra,\galgebra)$ can be regarded as a Lie algebra pair $(\galgebra \ltimes\thetaalgebra,\galgebra)$. The Atiyah class of a such Lie algebra pair   tells the nontrivially of this construction of Lie algebras \cite{MR3439229}.

Lie algebroid crossed modules, first introduced in \cite{0501103}, are a natural generalization of Lie algebra crossed modules, and actually, are the infinitesimal counterparts of crossed modules of Lie groupoids.  Existence of certain principal bundles are characterized by an obstruction
class of a certain crossed module of Lie algebroids  \cite{MR2415176}.

In an earlier work  {\cite{MR3035114}}, Chen and Xu, the notion of weak Lie 2-bialgebras and Lie bialgebra crossed modules.  It turns out that Lie bialgebra crossed modules are particular instances of weak Lie 2-bialgebras and have many interesting properties. For example, a Lie bialgebra crossed module can be lifted to a Poisson 2-group, and conversely, the infinitesimal of a Poisson 2-group is a Lie bialgebra crossed module \cite{MR3080481}.

 Therefore, as explained above, it is natural to quest for the proper generalized notion of Lie bialgebroid crossed modules. The answer we found is Definition \ref{Defn:bi-crossedmoduleofLiealgebra}. We shall see that this notion  involves  Lie bialgebroids, Courant algebroids and Lie algebroid crossed modules. Moreover, it can be equivalently described by a matched pair of Lie algebroids. It is quite novel to see both a matched pair of Lie algebroids and a Courant algebroid appearing in such a structure.

Moreover, we wish to find the Manin triple language of Lie bialgebroid crossed modules.
In this paper, we will propose our solutions to this question. We shall consider a pair of Lie algebroid crossed modules in duality, $\crossedmoduletriple{\thetaalgebra}{\phi}{\galgebra} $ and $\crossedmoduletriple{\galgebrastar}{\phiUprotate}{\thetaalgebrastar} $, where $\phiUprotate=-\phi^*$. They form a Lie bialgebroid crossed module if $(A_{\galgebra\triangleright\theta},A_{\theta^*\triangleright\galgebra^*}
 )$ is a Lie bialgebroid (see Definition \ref{Defn:bi-crossedmoduleofLiealgebra}).
We also define the notion of co-quadratic Lie algebroids, as well as their Dirac structures (see Definitions \ref{Defn:coquadratic} and  \ref{Defn:DiracStrs}). Accordingly, a co-quadratic Manin triple $(K,P,Q)$ is a co-quadratic Lie algebroid $K$ together with a pair of transverse Dirac structures $P$ and $Q$ (see Definition \ref{Defn:coquadraticManintriple}).

Our main results are the following: Theorem \ref{Thm:mathchedpairiff} ~gives an equivalent description of Lie bialgebroid crossed modules in the language of matched pairs. Theorem \ref{Thm:crossedmoduledouble} establishes a one-one correspondence between crossed modules of Lie bialgebroids and co-quadratic Manin triples.

This paper is organized as follows. Section 2 contains a succinct account of standard facts about Lie algebroid crossed modules, Lie algebroid matched pairs and Lie bialgebroids. Section 3 states the definition of Lie bialgebroid crossed modules and the main theorems on Lie biaglebroid crossed modules. Our results are then particularized to a few concrete examples in Section 4. Finally, Section 5 establishes a list of important identities and lemmas, and subsequently gives the proofs of our main theorems.

Notation: the pairing of a vector $x\in V$ and a co-vector $\xi\in V^*$ is written as any of the following forms: $\xi(x)$,  $x(\xi)$, $\pairing{x}{\xi}$  or $\pairing{\xi}{x}$.

\textbf{Acknowledgments}\\
We would like to express our gratitude to the following institutions for their hospitality while we
were working on this project: Tsinghua University (Yin) and Penn State University
(Lang). We would also like to thank Professor Zhangju Liu and Professor Mathieu Stienon for fruitful discussions and useful suggestions.


\section{Preliminaries}


\subsection{Lie algebroid crossed modules}
We first review the notion of Lie algebroids, for which the reader may refer to \cite{MR2157566}.
\begin{defn}
A Lie algebroid $(A,\lb{\cdot}{\cdot},\anchor)$ over $M$ consists of a vector bundle $A\to M$, a bundle map $\anchor:A\to TM$ called anchor and a Lie algebra bracket $\lb{\cdot}{\cdot}$ on the space of sections $\sections{A}$ such that $\anchor$ induces a Lie algebra homomorphism from $\sections{A}$ to $\XX(M)$ and the Leibniz identity \begin{equation*} \lb{X}{fY}=\big(\anchor(X)f\big)Y+f\lb{X}{Y} \end{equation*}
is satisfied for all $f\in\cinf{M}$ and $X,Y\in\sections{A}$.
\end{defn}

Every Lie algebra can be regarded as a Lie algebroid whose underlying base manifold is the one-point space. The notion of Lie algebroids  is  certainly a combination of the particular case of Lie algebras and the tangent bundle $TM$ of a manifold $M$.

It is well-known that a Lie algebroid $(A,\ba{\cdot}{\cdot},\anchor)$ gives rise to a Gerstenhaber algebra $(\sections{\wedge\graded A},\wedge,\lb{\cdot}{\cdot})$,
and a degree~1 derivation $\dee$ of the graded commutative algebra $(\sections{\wedge\graded {\As}},\wedge)$ such that $\dee^2=0$ \cite{MR1675117}. Here the (Lie algebroid) differential $\dee$ is given by
\begin{multline*}
(\dee\alpha)(X_1,\cdots,X_n,X_{n+1})=\sum_{i=1}^{n+1} (-1)^{i+1} \anchor(X_i)
 \alpha(X_1,\cdots,\widehat{X_i},\cdots,X_n,X_{n+1}) \\
+ \sum_{i<j} (-1)^{i+j} \alpha(\ba{X_i}{X_j},X_1,\cdots,\widehat{X_i},\cdots,\widehat{X_j},\cdots,X_{n+1}).
\end{multline*}
To each $X\in\sections{A}$, there is an associated degree~$(-1)$ derivation $\ii{X}$ of the graded commutative algebra $(\sections{\wedge\graded {\As}},\wedge)$, given by
\begin{equation*} (\ii{X}\alpha)(X_1,\cdots,X_n)
=\alpha(X,X_1,\cdots,X_n) .\end{equation*}
The Lie derivative $\ld{X}$ in the direction of a section $X\in\sections{A}$ is a degree~0 derivation of the graded commutative algebra $(\sections{\wedge\graded {\As}},\wedge)$ defined by the relation $\ld{X}=\ii{X}\dee+\dee\ii{X}$.

To avoid possible confusions, we sometimes write $\ld{X}^{A}$ instead of $\ld{X}$ to
indicate that the Lie derivative is coming from the Lie algebroid structure of $A$.

We remark that  $\Ker\anchor\subset A$, though it might be singular, is a bundle of Lie algebras, and known as the isotropic bundle of $A$. The Lie bracket defined pointwisely is induced from that of $\sections{A}$.

 \begin{ex}\label{Ex:actionLiealgebroids}
 An action of a Lie algebra $\galgebra$ on a manifold $M$ is a Lie algebra homomorphism $\rho:\galgebra\rightarrow \mathfrak{X}(M)$. Such an action gives rise to a Lie algebroid structure on the trivial bundle $M\times \galgebra\longrightarrow M$ whose anchor is induced from $\rho$, and the bracket on $\sections{M\times \galgebra}$ is given by
 $$
 \lb{fu }{gv}=fg\lb{u}{v}_{\galgebra}+f\bigl(\rho(u)g\bigr)v-g
 \bigl(\rho(v)f
 \bigr)u,
 $$
 where $u,v\in\galgebra$, $f,g\in \cinf{M}$.

 This Lie algebroid is called the action Lie algebroid and will be denoted by $M^{\galgebra}$ in the sequel. More details can be found in \cite{MR2157566}.
\end{ex}

We shall need the notion of Lie algebroid crossed modules \cite{0501103}.\footnote{In \cite{0501103}, it is supposed that $\galgebra$ is a transitive Lie algebroid. In this paper, it is unnecessary to maintain such a prerequisite.}
\begin{defn}
A Lie algebroid crossed module over a manifold $M$ is a quadruple $(\thetaalgebra,\phi,\galgebra,\triangleright)$, where $\thetaalgebra $ is a Lie algebra bundle over $M$, $\galgebra\longrightarrow M$ is a Lie algebroid, $\phi:\thetaalgebra\longrightarrow \galgebra$ is a {Lie algebroid morphism}
and $\triangleright:\galgebra\longrightarrow \mathfrak{D}(\thetaalgebra)$ is a representation
of $\galgebra$ on $\thetaalgebra$, such that
\begin{itemize}
\item[\rm(1)] $\phi(u)\triangleright v=[u,v]$, for all $u,v\in \sections{\thetaalgebra}$;
\item[\rm(2)] $\phi(x\triangleright u)=[x,\phi(u)]$, for all $x\in \sections{\galgebra}, u\in \sections{\thetaalgebra}.$
\end{itemize}
\end{defn}

\begin{rmk}
From the definition, it is easy to see that $\rho_{\galgebra}\circ \phi=0$, where $\rho_{\galgebra}$ is the anchor of of $\galgebra$. In other words, $\phi$ takes values in the isotropy Lie algebra bundle of $\galgebra$.
\end{rmk}


In this paper, we also use the shorthand notation $\crossedmoduletriple{\thetaalgebra}{\algebracrossedmodulemap}{\galgebra}$ to
denote such a Lie algebroid crossed module.

\begin{ex}[\cite{0501103}]\label{Ex:crossedmodulefromsymplectic}
{Let $(M,\omega)$ be a symplectic manifold.
Then the vector bundle $\galgebra=TM\oplus (M\times
\mathbb{R})$  becomes a  Lie algebroid over $M$ whose anchor is the
projection to the first component, and Lie bracket is
$$[X+f,Y+g]=[X,Y]+X(g)-Y(f)-\omega(X,Y),$$
where $X,Y\in \XX(M)$, $f,g\in \cinf{M}$.

Let $\theta=M\times \mathbb{R}$, $\phi$ be the natural inclusion of $\thetaalgebra$ into $\galgebra$. Define an action of $\galgebra$ on $\theta$ by $(X+f)\triangleright g=X(g)$. Then $(\theta,\galgebra,\phi,\triangleright)$ is a Lie algebroid crossed module.}
\end{ex}

\begin{lem}\label{Lem:ensentialcrossedmoduleLiealgebra}
Given a Lie algebroid $\galgebra$, a  $\galgebra$-module
$\thetaalgebra$, and  a bundle map $\algebracrossedmodulemap:~\thetaalgebra\lon
\galgebra$ satisfying the following two conditions:
\begin{eqnarray*} &\algebracrossedmodulemap (  x \moduleaction u)=\ba{ x }{\algebracrossedmodulemap
(u)},\\  &{\algebracrossedmodulemap (u)} \moduleaction v=-{\algebracrossedmodulemap (v)}
\moduleaction  u ,
\end{eqnarray*}for all $ ~ u,v\in
\sections{\thetaalgebra},  x \in \sections{\galgebra}$, there exists a unique Lie
algebra bundle structure on $\thetaalgebra$ such that
$\crossedmoduletriple{\thetaalgebra}{\algebracrossedmodulemap}{\galgebra} $ is a Lie
algebroid crossed module.
\end{lem}
\begin{proof}
Define the Lie bracket on $\theta$ by
$\ba{u}{v}=\algebracrossedmodulemap (u)\moduleaction v$, $\forall u,v\in\sections{\thetaalgebra}$. The rest of the claim can be verified directly.
 \end{proof}

The following fact is needed, which can be verified directly.
\begin{prop}\label{Prop:lie2algebra}
Given a Lie algebroid crossed module
$\crossedmoduletriple{\thetaalgebra}{\algebracrossedmodulemap}{\galgebra} $, there
exists a Lie algebroid $A_{\galgebra\triangleright \theta}$, which is the
direct sum $\galgebra\oplus \thetaalgebra$ equipped  with
the  {anchor $\rho=\rho_\galgebra$ and the Lie
bracket such that the  $\thetaalgebra$ and
$\galgebra$ are Lie subalgebroids}, and
 $$
 \ba{x}{ u }= x \moduleaction u,\quad\forall   x \in \sections{\galgebra}, u\in \sections{\thetaalgebra}.
 $$
\end{prop}

\subsection{Lie algebroid matched pairs}

If $P$ and $Q$ are two Lie subalgebroids of a Lie algebroid $L$
such that $L=P\oplus Q$  as vector bundles, then $L/P\isomorphism Q$
is naturally a $P$-module while $L/Q\isomorphism P$ is naturally a $Q$-module.
In this situation, the Lie algebroids $P$ and $Q$ are said to form a matched pair.

 \begin{defn}[\cite{MR1430434,MR1460632,MR1716681}]\label{Defn:mathchedpair}
A pair $(P,Q)$ of two   Lie algebroids over the same base manifold $M$,
 with anchors $\rho_P$ and $\rho_Q$ respectively, is called  a matched pair
if there exists  {a representation}   of $P$ on $Q$ and {a representation} of $Q$ on $P$, both denoted by $\triangleright$,
such that the identities
\begin{gather*}
\ba{\rho_P (X)}{\rho_Q  (Y)} = -\rho_P \big(Y\triangleright X\big)+\rho_Q  \big(X\triangleright Y\big) , \\
X\triangleright\ba{Y_1}{Y_2} = \ba{X\triangleright Y_1}{Y_2} +
\ba{Y_1}{X\triangleright Y_2} + ({Y_2}\triangleright X)\triangleright Y_1 -
({Y_1}\triangleright X)\triangleright Y_2, \\
Y\triangleright\ba{X_1}{X_2} = \ba{Y\triangleright X_1}{X_2} +
\ba{X_1}{Y\triangleright X_2} + (X_2\triangleright Y)\triangleright X_1 -
(X_1\triangleright Y)\triangleright X_2,
\end{gather*}
hold for all $X_1,X_2,X\in\sections{P}$ and $Y_1,Y_2,Y\in\sections{Q}$.
\end{defn}
The following lemma is standard.
\begin{lem}[\cite{MR1460632,MR1716681}]\label{lem:matchedpair}
Given a matched pair $(P,Q)$ of Lie algebroids, there is a Lie
algebroid structure $P\bowtie Q$ on the direct sum vector bundle
$P\oplus Q$, with anchor \[ X\oplus Y\mapsto\rho_P(X)+\rho_Q(Y) \] and bracket
\begin{equation*}
\ba{X_1\oplus Y_1}{X_2\oplus Y_2}
= \big(\ba{X_1}{X_2} + Y_1\triangleright X_2 - Y_2\triangleright X_1 \big)
\oplus \big( \ba{Y_1}{Y_2} + X_1\triangleright Y_2 - X_2\triangleright Y_1 \big) .
\end{equation*}
Conversely, if $P\oplus Q$ has a Lie algebroid structure for
which $P\oplus 0$ and $0\oplus Q$ are Lie subalgebroids, then the
representations defined by
\[ \ba{X\oplus 0}{0\oplus Y} = -Y \triangleright X\oplus X\triangleright Y \]
endow the couple $(P,Q)$ with a matched pair structure.
\end{lem}
{
\begin{ex}[\cite{MR1460632}]
A right $\galgebra$-action $\phi:\galgebra\longrightarrow \mathfrak{X}(M)$ on a manifold $M$ gives rise to an action Lie algebroid $M^\galgebra$ (see Example \ref{Ex:actionLiealgebroids}) and a matched pair of Lie algebroids $(TM,M^{\galgebra})$. The two representations are defined by
$$ X\triangleright (fx)=X(f)x,\ \ \ \ \ \ \ \ \ \ (fx)\triangleright X=f[\phi(x),X],$$
for any $X\in \mathfrak{X}(M), f\in C^\infty(M)$ and $x\in \galgebra$.
\end{ex}
}

{The following example is well-known.}
\begin{ex}[\cite{MR2685337}]\label{ex:liebialgebra}
If the direct sum $\mfg\oplus\mfg^*$ of a vector space $\mfg$ and its dual $\mfg^*$
is endowed with a Lie algebra structure such that the direct summands $\mfg$ and $\mfg^*$ are Lie subalgebras and
\[ [X,\alpha]=\coadjoint{X}\alpha - \coadjoint{\alpha}X
,\qquad \forall X\in\mfg, \alpha\in\mfg^* ,\]
the pair $(\mfg,\mfg^*)$ is said to be a Lie bialgebra.
Lie bialgebras are instances of matched pairs of Lie algebroids. The Lie algebra $\mfg\bowtie\mfg^*$ is also called the double Lie algebra arising from the given Lie bialgebra.
\end{ex}

\subsection{Lie bialgebroids, Courant algebroids and Manin triples}

In this section, we state basic definitions and facts about Lie bialgebroids and Courant algebroids.

Let $A\to M$ be a vector bundle. Assume that $A$ and its dual ${\As}$ are both Lie algebroids with anchor maps
$\anchor:A\to TM$ and $\anchors:{\As}\to TM$, brackets on sections
$\sections{A}\otimes_{\reals}\sections{A}\to\sections{A}:u\otimes v\mapsto\ba{u}{v}$ and
$\sections{{\As}}\otimes_{\reals}\sections{{\As}}\to\sections{{\As}}:\theta\otimes \eta\mapsto\bas{\theta}{\eta}$, and differentials
$\dee:\sections{\wedge^{\bullet}{\As}}\to\sections{\wedge^{\bullet+1}{\As}}$
and $\dees:\sections{\wedge^{\bullet}A}\to\sections{\wedge^{\bullet+1}A}$, respectively.

\begin{defn}[\cites{MR1362125,MR1746902,MR1262213}]
A pair of Lie algebroids $(A,\As)$ as above is a Lie bialgebroid    if $\dees$ is a
derivation of the Gerstenhaber algebra $(\sections{\wedge\graded
A},\wedge,\ba{\cdot}{\cdot})$, i.e.,
$$\dees\ba{u}{v}=\ba{\dees u}{v} +(-1)^{k-1}\ba{u}{\dees v},\quad\forall u\in\sections{\wedge^k A}, \forall v\in\sections{\wedge^l
A} .$$

Or, equivalently, if $\dee$ is a
derivation of the Gerstenhaber algebra $(\sections{\wedge\graded
{\As}},\wedge,\bas{\cdot}{\cdot})$, i.e.,
     $$\dee\bas{\theta}{\eta}=\bas{\dee\theta}{\eta}+(-1)^{k-1}\bas{\theta}{\dee\eta},\quad\forall \theta\in\sections{\wedge^k {\As}}, \forall \eta\in\sections{\wedge^l
{\As}} .$$
\end{defn}

Since the bracket
$\bas{\cdot}{\cdot}$ (resp. $\ba{\cdot}{\cdot}$) can be recovered
from the derivation $\dees$ (resp. $\dee$), one is led to the
following alternative definition: a Lie bialgebroid is a pair
$(A,\dees)$ consisting of a Lie algebroid
$(A,\ba{\cdot}{\cdot},\anchor)$ and a degree~1 derivation $\dees$ of
the Gerstenhaber algebra $(\sections{\wedge\graded
A},\wedge,\ba{\cdot}{\cdot})$ such that $\dees^2=0$.

In the particular case that $M$ is a single point, a Lie bialgbroid $(A,\As)$ degenerates to a Lie bialgebra (see Example \ref{ex:liebialgebra}).

 \begin{ex}\label{ex:classicalrmatrix}
Let us briefly recall the notion of exact Lie bialgebroids
\cite{MR1371234,MR1362125}. Let $A$ be a Lie algebroid with bracket $\ba{\cdot}{\cdot}$ on $\sections{A}$ and anchor map $\anchor:A\to TM$.
Given $\Lambda\in\sections{\wedge^2 A}$ satisfying
$\ba{\ba{\Lambda}{\Lambda}}{X}=0$ for all $X\in\sections{A}$, the bracket
\begin{equation*}
\Lambdabracket{\xi,\theta}=\LieDer_{\Lambdasharp(\xi)}\theta-\LieDer_{\Lambdasharp(\theta)}\xi
-\dA (\Lambda(\xi,\theta))=\LieDer_{\Lambdasharp(\xi)}\theta
-\inserts_{\Lambdasharp(\theta)}\dA \xi
\end{equation*}
on $\sections{\As}$ and the anchor map
$\anchors=\anchor\rond\Lambda\diese$,
make ${\As}$ a Lie algebroid.
The pair of Lie algebroid structures on $A$ and ${\As}$ fits into a Lie bialgebroid $(A,\As)$, which is known as an exact Lie bialgebroid. Such an element $\Lambda$ is also known as an $r$-matrix on $A$.
\end{ex}

A closely related notion to that of Lie bialgebroids is that of Courant algebroids, which is introduced in \cite{MR1472888} as a way to merge the concept of Lie bialgebras and the bracket on $\XX(M)\oplus\OO^1(M)$ --- here $M$ is a smooth manifold --- first discovered by Courant \cite{MR998124}. Roytenberg gave an equivalent definition phrased in terms of the Dorfman bracket \cite{MR2699145}.

\begin{defn}\label{Defn:CA}
A Courant algebroid consists of a vector bundle $\pi:E\to M$, a non-degenerate pseudo-metric $\ip{\cdot}{\cdot}$ on the fibers of $\pi$,
a bundle map $\rho:~E\to TM$ called anchor, and an $\reals$-bilinear
operation $\db{}{}$ on $\sections{E}$ called Dorfman bracket, which,
for all $f\in\cinf{M}$ and $x,y,z\in\sections{E}$, satisfy the
relations
\begin{eqnarray}\label{CA:1}
 \db{x}{(\db{y}{z})}&=&\db{(\db{x}{y})}{z}+\db{y}{(\db{x}{z})}; \\\label{CA:2}
 \rho(\db{x}{y})&=&\lb{\rho(x)}{\rho(y)};  \\\label{CA:3}
  \db{x}{fy}&=&\big(\rho(x)f\big)y+f(\db{x}{y});  \\\label{CA:4}
 \db{x}{y}+\db{y}{x}&=&\DD\ip{x}{y};  \\\label{CA:5}
 \db{\DD f}{x}&=&0;  \\\label{CA:6}
  \rho(x)\ip{y}{z}&=&\ip{\db{x}{y}}{z}+\ip{y}{\db{x}{z}}
,  \end{eqnarray}
where $\DD:\cinf{M}\to\sections{E}$ is the $\reals$-linear map defined by
$\ip{\DD f}{x}=\rho(x)f$.
\end{defn}

The symmetric part of the Dorfman bracket is given by \eqref{CA:4}.
The Courant bracket is defined as the skew-symmetric part
$\cb{x}{y}=\thalf(\db{x}{y}-\db{y}{x})$ of the Dorfman bracket. Thus we have the relation $\db{x}{y}=\cb{x}{y}+\thalf\DD\ip{x}{y}$.

The definition of a Courant algebroid can be rephrased using the Courant bracket instead of the Dorfman bracket \cite{MR1472888}.  More, regular Courant algebroids are extensively studied in
\cite{MR3022918}.

\begin{defn}
A Dirac structure of a Courant algebroid $E$ is a smooth subbundle $D\to M$, which is maximally isotropic with respect to the pseudo-metric and whose space of sections is closed under the Dorfman bracket.
\end{defn}

Thus a Dirac structure inherits a canonical Lie algebroid structure.

In analogue to Example \ref{ex:liebialgebra},
the notion of Courant algebroids  plays the role of the double of Lie bialgebroids. In fact, the relation between Courant algebroids and Lie bialgebroids is given  {as follows}.

\begin{thm}[\cite{MR1472888}]
\label{Thm:double} There is a one-one correspondence between Lie
bialgebroids and pairs of transverse Dirac structures of a Courant
algebroid.
\end{thm}

More precisely,   if the pair $(A,{\As})$ is a Lie bialgebroid,
then the vector bundle $A\oplus {\As}\to M$ together with the pseudo-metric
\begin{equation*}  \pairing{X_1+\xi_1}{X_2+\xi_2}=\xi_1(X_2)+\xi_2(X_1),\end{equation*}
the anchor map $\rho=\anchor+\anchors$ (whose dual is given by $\DD f=\dee f+\dees f$ for $f\in\cinf{M}$, and the Dorfman bracket
\begin{equation}\label{Eqt:Dorfmanbracket} \db{(X_1+\xi_1)}{(X_2+\xi_2)}=\big(\ba{X_1}{X_2}+\ld{\xi_1}X_2-\ii{\xi_2} \dees X_1 \big)+\big(\bas{\xi_1}{\xi_2}+\ld{X_1}\xi_2-\ii{X_2} \dee\xi_1 \big) \end{equation}
forms a Courant algebroid in which $A$ and ${\As}$ are transverse Dirac structures, where $X_1,X_2$ denote arbitrary sections of $A$ and $\xi_1,\xi_2$ are arbitrary sections of ${\As}$.
This Courant algebroid is called the double of the Lie bialgebroid $(A,{\As})$.


Following Liu-Xu-Weinstain \cite{MR1472888}, we have
\begin{defn}The data
$(E,A,B)$ is called a Manin triple, where $E$ is a Courant algebroid,  $A$ and $B$ are mutually transverse Dirac structures  {of $E$}.
\end{defn}

\begin{ex}Assume that we are given a Lie algebroid matched pair $(P,Q)$ as in Definition \ref{Defn:mathchedpair}.
The action of $P$ on $Q$ induces an action of $P$ on $Q^*$:
\begin{equation*}
\pairing{X\triangleright \xi}{Y}=\rho_{P}(X)\pairing{\xi}{Y}-\pairing{\xi}{X\triangleright Y},\ \ \ \ \forall X\in \Gamma(P), Y\in \Gamma(Q), \xi\in \Gamma(Q^*).
\end{equation*}
Similarly, there is an induced action of $Q$ on $P^*$.
Then one is able to construct a pair of semi-direct product Lie algebroids: $$A=P\ltimes Q^*,\qquad {\As}=Q\ltimes P^*. $$
Here $P^*\subset \As$ (and similarly for that of $Q^*\subset A$) is a trivial Lie subalgebroid. It turns out that they form a Lie bialgebroid.

In fact, The Lie algebroid $P\bowtie Q$ and its dual bundle $P^*\oplus Q^*$, as a trivial Lie algebroid, form a Lie bialgebroid. We thus have a Courant algebroid $E=P\bowtie Q\oplus P^*\oplus Q^*$, in which  $A$ and ${\As}$ are transverse Dirac structures. Hence we have the conclusion.
\end{ex}

\begin{ex}\label{action Lie bialgebroid}
Let $(\galgebra,\galgebra^*)$ be a Lie bialgebra, $\galgebra\bowtie \galgebra^*$ the double Lie algebra (see Example \ref{ex:liebialgebra}).
Suppose that there exist two actions $\varphi: \galgebra\longrightarrow \mathfrak{X}(M)$ and $\psi:\galgebra^*\longrightarrow \mathfrak{X}(M)$ such that $$[\varphi(x),\psi(\xi)]=\psi(ad_x^* \xi)-\varphi(ad_\xi^* x),\qquad\forall x\in \galgebra,\ \xi\in \galgebra^*.$$
We then
 get an action $\varphi\oplus\psi: \galgebra\bowtie \galgebra^*\longrightarrow \mathfrak{X}(M)$.
Assume further that $Ker(\varphi\oplus\psi)$ is coisotropic, i.e.,
$$\pairing{x_1}{\xi_2}+\pairing{x_2}{\xi_1}=0, \qquad\forall (x_i,\xi_i)\in \galgebra\oplus\galgebra^*, s.t.,\varphi(x_i)+\psi(\xi_i)=0.$$
We claim there is a Lie bialgebroid structure on $(M^\galgebra,M^{ \galgebra^*})$, where $M^\galgebra$ and $M^{\galgebra^*}$ are,  respectively, the action Lie algebroids induced by the actions $\varphi$ and $\psi$ (see Example \ref{Ex:actionLiealgebroids}).

To see this, we need a result in \cite{MR2507112},  which states that if there is an action $\phi$ of a quadratic Lie algebra $\galgebra$  on a manifold $M$ and  the kernel of $\phi$ is coisotropic, then there exists a Courant algebroid structure on $M\times \galgebra$ whose anchor is given by $\phi$.
In this example, we take $\galgebra\bowtie \galgebra^*$ as the quadratic Lie algebra, and get a Courant algebroid   $M\times(\galgebra\bowtie \galgebra^*)$. Moreover,  $M^\galgebra$ and $M^{\galgebra^*}$ are transverse Dirac structures, so they form a Lie bialgebroid.
\end{ex}

\section{Lie Bialgebroid Crossed Modules}
The  notion of Lie bialgebra crossed modules is   introduced in \cite{MR3035114}. The motivation is to study certain Poisson structures on semi-strict Lie 2-groups and it turns out that the infinitesimal data of such structures  exactly matches the notion of a Lie bialgebra crossed module.

A straightforward generalized notion is the following.

\begin{defn}\label{Defn:bi-crossedmoduleofLiealgebra}
 A Lie bialgebroid crossed module  is  a pair of Lie algebroid crossed modules in duality:
$\crossedmoduletriple{\thetaalgebra}{\phi}{\galgebra} $ and
$\crossedmoduletriple{\galgebrastar}{\phiUprotate
}{\thetaalgebrastar} $, where $\phiUprotate=-\phi^*$, such that
$(A_{\galgebra\triangleright\theta},A_{\theta^*\triangleright\galgebra^*}
 )$ is a Lie bialgebroid. Here $A_{\galgebra\triangleright\theta}$ and $A_{\theta^*\triangleright\galgebra^*}$ are  Lie algebroids  defined in Proposition \ref{Prop:lie2algebra}.

\end{defn}

Below are the main results of this paper, all of them will be proved in the last section.
The first theorem reveals the relation between matched pairs and Lie bialgebroid crossed modules.

\begin{thm}\label{Thm:mathchedpairiff}
Let $\crossedmoduletriple{\thetaalgebra}{\phi}{\galgebra} $ and
$\crossedmoduletriple{\galgebrastar}{\phiUprotate
}{\thetaalgebrastar}$ both be Lie algebroid crossed modules.  Then
they form a Lie bialgebroid crossed module if and only if the pair
$(\galgebra,\thetaalgebrastar)$ is a matched pair of Lie algebroids, where the  $\galgebra$-action on $\thetaalgebrastar$ is dual to the given $\galgebra$-action on $\thetaalgebra$,  {and} the  $\thetaalgebrastar$-action on $\galgebra$ is dual to the given $\thetaalgebrastar$-action on $\galgebrastar$. \end{thm}

We are in  {the} position to introduce a new notion.

\begin{defn}\label{Defn:coquadratic}
A  \textbf{co-quadratic} Lie algebroid is a Lie algebroid $K$ over $M$
together with a symmetric, bilinear (possibly degenerate) $2$-form
on the dual vector space:
$$
\bookpairing{\tobefilledin,\tobefilledin}:~\quad \Gamma(K^*)\times \Gamma(K^*)\lon
C^\infty(M),
$$
which is $K$-invariant, i.e., $\forall~ X\in \Gamma(K)$,
$\gamma,\gamma'\in\Gamma(K^*)$,
\begin{equation}\label{Eqt:bookpairinginvariant}
\rho(X)\bookpairing{\gamma,\gamma'}=
\bookpairing{\ld{X}^{K}\gamma,\gamma'}+\bookpairing{\gamma,\ld{X}^{K}\gamma'},
\end{equation}
where $\rho$ is the anchor of $K$.
\end{defn}

The following proposition can be proved easily.

\begin{prop}A co-quadratic Lie algebroid structure over the Lie algebroid $K$ is determined by an element $C\in K\symmetricproduct K$ (where $\symmetricproduct$ denotes the symmetric tensor product) which is $K$-invariant: $ \ld{X}^{K}C=0$, $\forall X\in\sections{K}$.
\end{prop}

\begin{defn}\label{Defn:DiracStrs}
Let $K$ be a co-quadratic Lie algebroid. A \textbf{Dirac structure} of $K$ is a  Lie subalgebroid $D\subset K$
such that the null space  $$D^0=\set{\gamma\in
 K^*_p~|~p\in M,\pairing{\gamma}{d}=0,\forall d\in K_p}$$ is isotropic with respect to
$\bookpairing{\tobefilledin,\tobefilledin}$.
\end{defn}
\begin{defn}\label{Defn:coquadraticManintriple}
We call $(K,P,Q)$ a co-quadratic Manin triple if $K$ {is} a co-quadratic Lie algebroid, $P$ and $Q$ are two Dirac structures of $K$  which are transverse to
each other, i.e., $K=P\oplus Q$.
\end{defn}
The following theorem is the second main result of this paper, which can be regarded as the crossed module version of Manin triples.

\begin{thm}
\label{Thm:crossedmoduledouble} There is a one-one correspondence between Lie
bialgebroid crossed modules and co-quadratic Manin triples.
\end{thm}

This theorem is reformulated by  the following two Propositions{:} \ref{Prop:manin3} and \ref{Prop:manin3 reverse}.

\begin{prop}\label{Prop:manin3} Given a co-quadratic Manin triple $(K,P,Q)$, there  {is} a Lie bialgebroid
crossed module structure on the dual crossed modules
$\crossedmoduletriple{P^0 }{\phi}{P}$ and $\crossedmoduletriple{Q^0
}{\phiUprotate}{Q}$, where $\phi:~P^0\lon P$ is determined by
\begin{equation*}
\pairing{\xi}{\phi(u)}\defbe
\bookpairing{\xi,u},\quad~\forall \xi\in \Gamma(Q^0),u\in \Gamma(P^0).
\end{equation*}
(We treat $P^0=Q^*$ and $Q^0=P^*$ in the standard way.)
\end{prop}

For the converse of Proposition \ref{Prop:manin3}, we have
\begin{prop}\label{Prop:manin3 reverse} Given a Lie bialgebroid crossed module structure on the
dual crossed modules
$\crossedmoduletriple{\thetaalgebra}{\phi}{\galgebra} $ and
$\crossedmoduletriple{\galgebrastar}{\phiUprotate
}{\thetaalgebrastar}$, there  {is} a co-quadratic Lie algebroid
$$
K=\galgebra\bowtie \thetaalgebrastar
$$
with the $2$-form $\bookpairing{\tobefilledin,\tobefilledin} $ on
$K^*=\galgebrastar\oplus\thetaalgebra$ defined by
$$
\bookpairing{\xi_1+
u_1,\xi_2+u_2}=\pairing{\xi_1}{\phi(u_2)}+\pairing{\xi_2}{\phi(u_1)},\qquad\forall
\xi_1,\xi_2\in\Gamma(\galgebrastar),u_1,u_2\in\sections{\thetaalgebra}.
$$
Moreover, $\galgebra$ and $\thetaalgebrastar$ are transverse Dirac
structures in $K$.
\end{prop}

 The following facts are direct consequences.

\begin{cor}Any Lie bialgebroid crossed module structure of the form
$\crossedmoduletriple{\thetaalgebra}{0}{ \galgebra}$ is equivalent to
assign to $(\galgebra,\thetaalgebrastar)$ a Lie algebroid matched pair structure.
\end{cor}

 \begin{cor}\label{Cor:hinvariant} Given a matched pair $(P,Q)$ as in Definition \ref{Defn:mathchedpair}, and an element $h\in \sections{P\otimes Q}$ which is invariant:
\begin{equation*}
\ba{l}{h}=0,\quad\forall l\in \Gamma(P\bowtie Q),
\end{equation*}
there exists a Lie bialgebroid crossed module structure on the dual crossed modules
$\crossedmoduletriple{Q^* }{\phi}{P}$ and $\crossedmoduletriple{P^*
}{\phiUprotate}{Q}$, where $\phi:~Q^*\lon P$ is determined by
\begin{equation*}
 {\phi(\beta)}\defbe
\inserts_{\beta}h,\quad~\forall \beta\in \Gamma(Q^*).
\end{equation*}
\end{cor}

Furthermore, we have
\begin{cor}
A Lie bialgebroid crossed module structure on the
dual crossed modules
$\crossedmoduletriple{\thetaalgebra}{\phi}{\galgebra} $ and
$\crossedmoduletriple{\galgebrastar}{\phiUprotate
}{\thetaalgebrastar}$ induces two Lie bialgebroid structures $(A_{\galgebra\triangleright\theta},A_{\theta^*\triangleright\galgebra^*}
 )$, $(\galgebra\bowtie\theta^*,\galgebra^*\oplus\theta)$.
\end{cor}
\begin{rmk}In fact, the pairs $(\galgebra,\galgebra^*)$  and $(\theta,\theta^*)$ are also Lie bialgebroids. However, in the Courant algebroid $A_{\galgebra\triangleright\theta}\oplus A_{\theta^*\triangleright\galgebra^*}$
(see Proposition \ref{Prop:D'sdoublebracket}), the Dorfmann bracket on $\galgebra\oplus\galgebra^*$ is not closed, nor on $\theta\oplus\theta^*$.
\end{rmk}

\section{Examples}

In this section, we give several examples of Lie bialgebroid crossed modules.

 \begin{ex}
Let $A$ and $B$ be vector bundles over $M$, $d: A{\lon}
 B$ a bundle map. Equipped with
   abelian Lie algebroids on both $A$ and $B$, and the trivial
 action of $B$ on $A$, then $\crossedmoduletriple{A}{d}{ B}$ is an obvious
 Lie algebroid crossed module.  By Theorem \ref{Thm:mathchedpairiff}, any Lie bialgebroid crossed module structure of the form  $\crossedmoduletriple{A}{d}{ B}$
is equivalent to assign a Lie algebroid crossed module structure to
$\crossedmoduletriple{B^*}{d^T}{ {\As}}$.

\end{ex}

\begin{ex}
We recall the action of a Lie algebra crossed module $\crossedmoduletriple{\theta}{\phi}{\galgebra}$ on a smooth manifold $M$ \cite{MR2996854}. Such an action is in fact a Lie algebra crossed module homomorphism $\varphi=(\varphi_0,0):\crossedmoduletriple{\theta}{\phi}{\galgebra}\longrightarrow\crossedmoduletriple{0}{}{\mathfrak{X}(M)}$. In other words, $\varphi_0: \galgebra\longrightarrow\mathfrak{X}(M)$ is a Lie algebra homomorphism such that $\varphi_0\circ \phi=0$. A Lie algebroid crossed module structure
 $\crossedmoduletriple{M\times \theta}{\phi}{M^\galgebra}$ is induced. Here ${M^\galgebra}$ is an action Lie algebroid $M^\galgebra$ (see Example \ref{Ex:actionLiealgebroids}) and the action of $M^\galgebra$ on $M\times \theta$ is induced by the action of $\galgebra$ on $\theta$.

Now suppose that $\crossedmoduletriple{\theta}{\phi}{\galgebra}$ and $\crossedmoduletriple{\galgebra^*}{\phiUprotate}{\theta^*}$ form a Lie bialgebra crossed module defined in the earlier work \cite{MR3035114}.  Assume further that $\psi_0:\theta^*\longrightarrow\mathfrak{X}(M)$ defines an action of the Lie algebra crossed module $\crossedmoduletriple{\galgebra^*}{\phiUprotate}{\theta^*}$ on $M$.

If $\varphi_0$ and $\psi_0$ satisfy the following condition:
$$[\varphi_0(x),\psi_0(\alpha)]=\psi_0(x\triangleright \alpha)-\varphi_0(\alpha\triangleright x),\qquad\forall x\in \galgebra,\ \alpha\in \theta^*,$$ then the map $\varphi_0\oplus\psi_0: \galgebra\bowtie \theta^*\longrightarrow\mathfrak{X}(M)$ defines an action,
 and we have an action Lie algebroid
 ${M^{\galgebra\bowtie \theta^*}}$. It is easy to see that $(M^\galgebra, M^{\theta^*})$ is a matched pair and $M^\galgebra\bowtie {M^{\theta^*}}\cong {M^{\galgebra\bowtie \theta^*}}$. Thus, by Theorem \ref{Thm:mathchedpairiff}, $\crossedmoduletriple{M\times \theta}{\phi}{M^\galgebra}$ and $\crossedmoduletriple{M\times \galgebra^*}{\phiUprotate}{M^{\theta^*}}$ form a Lie bialgebroid crossed module.

\end{ex}

 \begin{ex}\label{ex:Rmatrix}   In Example \ref{ex:classicalrmatrix} we have reviewed the standard $r$-matrix construction of Lie bialgebroids. Below, we give an analogous construction of Lie bialgebroid crossed modules.

Let $\crossedmoduletriple{\thetaalgebra}{\phi}{\galgebra}$ be a Lie algebroid crossed module. We call $\Rmatrix\in \Gamma(\wedge^2\thetaalgebra)$
 \textbf{a crossed  module  $r$-matrix}
   if $\ba{\Rmatrix}{\Rmatrix}\in
\Gamma(\wedge^3\thetaalgebra)$ is $\galgebra$-invariant, i.e.,
$$
 x \moduleaction\ba{\Rmatrix}{\Rmatrix} =0,\quad\forall~  x \in
\sections{\galgebra}.
$$

It can be verified that  such a crossed module  $r$-matrix $\Rmatrix$ is an $r$-matrix on $\thetaalgebra$. Moreover, it gives rise to a Lie
 algebroid crossed module structure  $\crossedmoduletriple{\galgebrastar}{\phiUprotate}{
\thetaalgebrastar}$. The Lie algebroid structure of $\thetaalgebrastar$ is just
given by $\ba{\tobefilledin}{\tobefilledin}_{\Rmatrix}$ as in Example \ref{ex:classicalrmatrix}. And the
action of $\thetaalgebrastar$ on $\galgebrastar$ is given as follows:
$$
\pairing{\alpha\moduleaction\xi}{x}=-\pairing{ {\alpha}\wedge{\phiUprotate
(\xi) }}{x\moduleaction\Rmatrix},
$$
where $\alpha\in\Gamma(\thetaalgebrastar)$, $\xi\in\Gamma(\galgebrastar)$,
$x\in\sections{\galgebra}$.

Define $\mathbf{r}'\in \Gamma(\galgebra\wedge\thetaalgebra)\oplus\Gamma(\wedge^2 \galgebra)$  by the rule:
$$
\mbox{if}\ \ \mathbf{r}=\sum_i a_i\wedge b_i,~\ ~(a_i,b_i\in\Gamma(\thetaalgebra)), \quad\mbox{then}\quad\mathbf{r}'=\sum_i \bigl(~\phi(a_i)\wedge b_i+a_i\wedge \phi(b_i)-\phi(a_i)\wedge\phi(b_i)~\bigr).
$$
Then, $\mathbf{r}+\mathbf{r}'$ is an $r$-matrix of the Lie algebroid $A_{\galgebra\triangleright \theta}$. Moreover, one can check that the Lie algebroid structure on $A_{\theta^*\triangleright \galgebra^*}$ comes from $\mathbf{r}+\mathbf{r}'$. Hence
 $(A_{\galgebra\triangleright \theta},A_{\theta^*\triangleright \galgebra^*})$ is a Lie bialgebroid and by definition,   the Lie algebroid crossed modules $\crossedmoduletriple{\thetaalgebra}{\phi}{\galgebra}$ and
$\crossedmoduletriple{\galgebrastar}{\phiUprotate}{
\thetaalgebrastar}$ form a Lie bialgebroid crossed module.
\end{ex}

\section{Techniques and Proofs of Main Theorems}

 Throughout this section, we  suppose that $x,y\in\sections{\galgebra}$, $\xi,\eta\in \Gamma(\galgebra^*)$, $u,v\in \Gamma(\theta)$ and $\alpha,\beta\in \Gamma(\theta^*).$

\begin{prop}\label{prop:fact}Given a Lie algebroid crossed module
$\crossedmoduletriple{\thetaalgebra}{\phi}{\galgebra} $, we have the following equalities:  
\begin{eqnarray}
\label{Eqt:another}
\phiUprotate(\ld{x}^\galgebra\xi)&=&x\moduleaction \phiUprotate(\xi),\\
\label{Eqt:another2}\phi(u)\moduleaction \alpha&=&\ld{u}^{\theta}{\alpha}.
\end{eqnarray}
\end{prop}

\begin{proof}
For all $u\in \sections{\thetaalgebra}$, one has
\begin{eqnarray*}
\pairing{\phiUprotate(\ld{x}^\galgebra\xi)}{u}&=&-\pairing{\ld{x}^\galgebra\xi}{\phi(u)}
=-\rho_{\galgebra}(x)\pairing{\xi}{\phi(u)}+\pairing{\xi}{[x,\phi(u)]}\\&=&-\rho_{\galgebra}(x)\pairing{\xi}{\phi(u)}+\pairing{\xi}{\phi(x\moduleaction u)}\\ &=&\rho_{\galgebra}(x)\pairing{\phiUprotate(\xi)}{u}-\pairing{\phiUprotate(\xi)}{x\moduleaction u}\\&=&\pairing{x\moduleaction \phiUprotate (\xi)}{u}.
\end{eqnarray*}
This proves \eqref{Eqt:another}. For \eqref{Eqt:another2}, we observe that
\begin{eqnarray*}
\pairing{\phi(u)\moduleaction \alpha}{v}=-\pairing{\alpha}{\phi(u)\moduleaction v}
=-\pairing{\alpha}{[u,v]}=\pairing{\ld{u}^{\theta} \alpha}{v},
\end{eqnarray*}
where we have used the facts that $\rho_{\galgebra}\circ \phi=0$ and $\theta$ is a Lie algebra bundle.
\end{proof}

\begin{prop}\label{Prop:D'sdoublebracket}Suppose that $\crossedmoduletriple{\thetaalgebra}{\phi}{\galgebra} $ and $\crossedmoduletriple{\galgebrastar}{\phiUprotate}{\thetaalgebrastar}$ are Lie algebroid crossed modules in duality and let
$$E=A_{\galgebra\triangleright\theta}\oplus A_{\theta^*\triangleright\galgebra^*}= \galgebra\oplus\thetaalgebra\oplus\thetaalgebrastar\oplus\galgebrastar$$ be endowed with the Dorfmann bracket $\circ$  in Eqt. \eqref{Eqt:Dorfmanbracket}. Then,
\begin{enumerate}
\item[(1)]  Restricted to $\galgebra\oplus\galgebrastar$, we have
    \begin{equation*}
    x\circ \xi=\ld{x}^{\galgebra}\xi-\ld{\xi}^{\galgebra^*}x+x\vee\xi,\qquad\forall x\in \sections{\galgebra},\xi\in\Gamma(\galgebrastar).
    \end{equation*}
    Here $x\vee\xi \in \sections{\thetaalgebra}$ is defined by
\begin{equation}
\label{Eqt:Vee1}
    \pairing{x\vee\xi}{\alpha}{=\pairing{\xi}{\alpha\triangleright x}}=\rho_{\theta^*}
    (\alpha)\pairing{\xi}{x}-\pairing{x}{\alpha\moduleaction \xi} ,\quad\forall \alpha\in \Gamma(\thetaalgebrastar).
    \end{equation}
\item[(2)] Restricted to  $\thetaalgebra\oplus\thetaalgebrastar$, we have
    \begin{equation*}
    \alpha\circ u=\ld{\alpha}^{\theta^*}u - \ld{u}^{\theta}\alpha+\alpha\vee u,\qquad\forall u\in \sections{\thetaalgebra},\alpha\in\Gamma(\thetaalgebrastar).
    \end{equation*}
    Here $\alpha\vee u \in \Gamma(\galgebrastar)$ is defined by
\begin{equation}
\label{Eqt:Vee2}
    \pairing{\alpha \vee u}{x}{=\pairing{u}{x\triangleright \alpha}}=\rho_{\galgebra}(x)\pairing{\alpha}{u}-\pairing{\alpha}{x\moduleaction u} ,\quad\forall x\in \sections{\galgebra}.
\end{equation}
\item[(3)] Restricted to $\galgebra\oplus\thetaalgebrastar$, we have
    $$x\circ \alpha=x\moduleaction \alpha-\alpha \moduleaction x,\qquad \forall x\in \sections{\galgebra},\ \alpha\in \Gamma(\thetaalgebrastar).
    $$
    \item[(4)]   Restricted to  $\galgebrastar\oplus\thetaalgebra$, the Dorfmann bracket $\circ$ vanishes.
    \end{enumerate}
\end{prop}
\begin{proof} For simplicity, we write $A=A_{\galgebra\triangleright\theta}$ and ${\As}=A_{\theta^*\triangleright\galgebra^*}$. We first show Claim (1).
In fact,
\begin{eqnarray*}
\pairing{x\circ\xi}{y+u+\alpha+\eta}&=&\pairing{\ld{x}^A \xi-i_\xi d_{{\As}}x }{y+u+\alpha+\eta}\\&=&\rho_{\galgebra}(x)\xi(y)-\pairing{\xi}{[x,y+u]_A}-d_{{\As}}x(\xi,\alpha+\eta)\\
&=&\pairing{\ld{x}^{\galgebra}\xi}{y}+\rho_{\theta^*}(\alpha)\xi(x)+\pairing{x}{[\xi,\alpha+\eta]_{{\As}}}\\ &=&\pairing{\ld{x}^{\galgebra}\xi}{y}-\pairing{\ld{\xi}^{\galgebra^*}x}{\eta}+\rho_{\theta^*}(\alpha)\xi(x)-\pairing{x}{\alpha\triangleright\xi}.
\end{eqnarray*}
Claim (2) can be derived analogously.

For Claim (3), one needs to observe that\begin{eqnarray*}
\pairing{x\circ\alpha}{y+u+\beta+\eta}&=&\pairing{\ld{x}^A\alpha-i_\alpha d_{{\As}}x} {y+u+\beta+\eta}\\&=&\rho_{\galgebra}(x)\alpha(u)-\pairing{\alpha}{x\triangleright u}-\rho_{\theta^*}(\alpha)\eta(x)+\pairing{x}{\alpha\triangleright \eta}\\ &=&\pairing{x\triangleright \alpha}{u}-
\pairing{\alpha\triangleright x}{\eta}.
\end{eqnarray*}
Finally to show Claim (4), it suffices to prove that $\galgebrastar$ and $\thetaalgebra$ are commuting with respect to $\circ$. In fact, 
\begin{eqnarray*}
\pairing{\xi\circ u}{x+v+\beta+\eta}&=&\pairing{\ld{\xi}^{{\As}}u-i_ud_A \xi}{x+v+\beta+\eta}
\\&=&-\pairing{u}{[\xi,\beta+\eta]_{{\As}}}-d_A \xi(u,x+v)\\&=&0+0=0. \end{eqnarray*}
\end{proof}

\begin{lem}
Maintain the assumptions in Proposition \ref{Prop:D'sdoublebracket} and suppose further that $(\galgebra,\theta^*)$ is a matched pair. Then we have
\begin{eqnarray}
\label{Eqt:anchor}\rho_{\galgebra}(\ld{\xi}^{\galgebra^*}x)&=&0,\\
\label{Eqt:temp1}
\ld{x}^{\galgebra}[\xi,\eta]&=&[\ld{x}^\galgebra\xi,\eta]+[\xi,\ld{x}^\galgebra\eta]-\ld{\ld{\xi}^{\galgebra^*}x}^{\galgebra}\eta
+\ld{\ld{\eta}^{\galgebra^*}x}^{\galgebra}\xi-d_{\galgebra}\pairing{\ld{\eta}^{\galgebra^*}x}{\xi},\\
\label{555} \ld{\xi}^{\galgebra^*}[x,y]&=&[\ld{\xi}^{\galgebra^*} x,y]+[x,\ld{\xi}^{\galgebra^*} y]+\ld{\ld{y}^{\galgebra}\xi}^{\galgebra^*}x-\ld{\ld{x}^{\galgebra}\xi}^{\galgebra^*}y,
\end{eqnarray}
for all $\xi,\eta \in \Gamma(\galgebrastar), x,y\in \sections{\galgebra}$.
\end{lem}

\begin{proof}
Applying Proposition  \ref{prop:fact} to the Lie algebroid crossed module $\crossedmoduletriple{\galgebrastar}{\phiUprotate}{\thetaalgebrastar}$, we have
\begin{eqnarray}\label{Eqt:Lie deri}
\ld{\xi}^{\galgebra^*}x=\phiUprotate(\xi)\triangleright x.
\end{eqnarray}
 Thus, by Eqt. \eqref{Eqt:another} and the facts that $(\galgebra,\theta^*)$ is a matched pair and $\rho_{\theta^*}\circ \phiUprotate=0$, we have
\begin{eqnarray*}
\rho_{\galgebra}(\ld{\xi}^{\galgebra^*}x)=\rho_{\galgebra}(\phiUprotate(\xi)\triangleright x)=
\rho_{\theta^*}(x\triangleright\phiUprotate(\xi))=\rho_{\theta^*}\phiUprotate(\ld{x}^{\galgebra}\xi)=0.
\end{eqnarray*}
To prove Eqt. (\ref{Eqt:temp1}), we calculate
\begin{eqnarray*}
\pairing{\ld{x}[\xi,\eta]}{y}&=&\rho_{\galgebra}(x)\pairing{[\xi,\eta]}{y}-\pairing{[\xi,\eta]}{[x,y]}\\
&=&\rho_\galgebra(x)\pairing{[\xi,\eta]}{y}+\pairing{\phiUprotate(\eta)\moduleaction \xi}{[x,y]}\\&=&\rho_\galgebra(x)\pairing{[\xi,\eta]}{y}-\pairing{\xi}{\phiUprotate(\eta) \moduleaction [x,y]}\\
&=&\rho_{\galgebra}(x)\pairing{[\xi,\eta]}{y}\\ &&-\pairing{\xi}{[\phiUprotate(\eta)\moduleaction x,y]+[x,\phiUprotate(\eta)\moduleaction y]+(y\moduleaction \phiUprotate(\eta))\moduleaction x-(x\moduleaction \phiUprotate(\eta))\moduleaction y}.
\end{eqnarray*}
The last step is due to the fact that $(\galgebra,\thetaalgebrastar)$ is a matched pair.
Below we evaluate each of the above terms:
\begin{eqnarray*}
-\pairing{\xi}{[\phiUprotate(\eta)\moduleaction x,y]}= \pairing{\xi}{[y,\ld{\eta} x]}=
\pairing{\ld{\ld{\eta} x} \xi}{y} \qquad (\mbox{by Eqt. \eqref{Eqt:anchor},\eqref{Eqt:Lie deri}}),
\end{eqnarray*}
\begin{eqnarray*}
-\pairing{\xi}{[x, \phiUprotate(\eta) \moduleaction y]}&=&\pairing{\ld{x} \xi}{\phiUprotate(\eta)\moduleaction y}-\rho_{\galgebra}(x)\pairing{\xi}{\phiUprotate(\eta)\moduleaction y}\\&=&-\pairing{\phiUprotate(\eta)\moduleaction (\ld{x} \xi)}{y}+\rho_{\galgebra}(x)\pairing{\phiUprotate(\eta)\moduleaction \xi}{y}\\ &=&\pairing{[\ld{x} \xi,\eta]}{y}-
\rho_{\galgebra}(x)\pairing{[\xi,\eta]}{y},
\end{eqnarray*}
\begin{eqnarray*}
-\pairing{\xi}{(y\moduleaction \phiUprotate (\eta))\moduleaction x}&=&\pairing{\phiUprotate (\ld{y} \eta)\moduleaction \xi}{x}\qquad (\mbox{by Eqt. \eqref{Eqt:another}},\ \rho_{\theta^*}\circ \phiUprotate=0 )\\&=&\pairing{[\ld{y} \eta,\xi]}{x}=\pairing{\ld{y} \eta}{\ld{\xi} x}\\&=&\rho_{\galgebra}(y)\pairing{\eta}{\ld{\xi} x}-\pairing{\eta}{[y,\ld{\xi} x]}\\ &=&-\pairing{d_{\galgebra}\pairing{\ld{\eta} x}{\xi}}{y}-\pairing{\ld{\ld{\xi} x}\eta}{y} \qquad (\mbox{by Eqt. \eqref{Eqt:anchor}}),
\end{eqnarray*}
\begin{eqnarray*}
\pairing{\xi}{(x\moduleaction \phiUprotate (\eta))\moduleaction y}&=&\pairing{\xi}{\phiUprotate (\ld{x} \eta)\moduleaction y}\qquad (\mbox{by Eqt. \eqref{Eqt:another}})\\&=&-\pairing{\phiUprotate (\ld{x} \eta)\moduleaction \xi}{y}=\pairing{[\xi,\ld{x}\eta]}{y}.
\end{eqnarray*}
Combining these equalities, one proves Eqt. (\ref{Eqt:temp1}).

To prove Eqt. (\ref{555}), we again note that $(\galgebra,\theta^*)$ is a matched pair, and
\begin{eqnarray*}
 \pairing{\ld{\xi}[x,y]}{\eta}&=&-\pairing{[x,y]}{[\xi,\eta]}=\pairing{\eta}{\phiUprotate(\xi)\triangleright[x,y]}\\ &=&
 \pairing{\eta}{[\phiUprotate(\xi)\triangleright x,y]+[x,\phiUprotate(\xi)\triangleright y]+(y\triangleright \phiUprotate(\xi))\triangleright x-(x\triangleright \phiUprotate(\xi))\triangleright y}\\ &=&
 \pairing{\eta}{[\ld{\xi} x,y]+[x,\ld{\xi} y]+\ld{\ld{y} \xi}x-\ld{\ld{x} \xi}y},
\end{eqnarray*}
where we have used Eqt. (\ref{Eqt:Lie deri}) and the fact that
$$\pairing{\eta}{(y\triangleright \phiUprotate(\xi))\triangleright x}=\pairing{\eta}{(\phiUprotate(\ld{y} \xi)\triangleright x}=-\pairing{[\ld{y} \xi,\eta]}{x}=\pairing{\eta}{\ld{\ld{y} \xi}x}.$$
We thus get Eqt. (\ref{555}).

\end{proof}
\begin{proof}[Proof of Theorem \ref{Thm:mathchedpairiff}]
``$\Longrightarrow$'':\ \ If $\crossedmoduletriple{\thetaalgebra}{\phi}{\galgebra} $ and $\crossedmoduletriple{\galgebrastar}{\phiUprotate}{\thetaalgebrastar}$ form a Lie bialgebroid crossed module, then $A=A_{\galgebra\triangleright \theta}$ and $\As=A_{\theta^*\triangleright \galgebra^*}$ constitute a Lie bialgebroid. Thus, $$E=A \oplus \As= \galgebra\oplus\thetaalgebra\oplus\thetaalgebrastar\oplus\galgebrastar$$ is endowed with a Courant algebroid structure. By Proposition \ref{Prop:D'sdoublebracket}, $\galgebra\oplus\thetaalgebrastar$ is a Dirac structure, as well as a Lie algebroid with $\galgebra$ and $\theta^*$ as Lie sub-algebroids.
Hence by Lemma \ref{lem:matchedpair}, the pair $(\galgebra,\thetaalgebrastar)$ must be a matched pair.

``$\Longleftarrow$'':\ \ Assume that $(\galgebra,\thetaalgebrastar)$ is a matched pair. For $E=A\oplus {\As}$, we first check Conditions \eqref{CA:2} -- \eqref{CA:6} in Definition \ref{Defn:CA}.

    For any two Lie algebroids $A$ and $A^*$, one may directly verify that the Dorfman bracket (\ref{Eqt:Dorfmanbracket})  satisfies \eqref{CA:3},\eqref{CA:4} and \eqref{CA:6}. Since the anchor maps vanish on  $\theta$ and $\galgebra^*$, we have
    $$\DD f=d_* f+df \in \Gamma(\theta)\oplus\Gamma(\galgebra^*),$$  $\rho\circ \DD=0$, and to verify \eqref{CA:2}, we only need to prove the following equality:
\begin{equation}\label{Eqt:sdfdfsdf}\rho((\alpha+\xi)\circ (x+u))=[\rho_{\theta^*}(\alpha),\rho_{\galgebra}(x)].\end{equation}

In fact, by Proposition \ref{Prop:D'sdoublebracket} and (\ref{Eqt:anchor}), the LHS of \eqref{Eqt:sdfdfsdf} equals $-\rho_{\theta^*}(x\triangleright \alpha)+\rho_{\galgebra}(\alpha\triangleright x)$. Thus the two sides of \eqref{Eqt:sdfdfsdf} match because $(\galgebra,\theta^*)$ is a matched pair.

To show \eqref{CA:5}, we check that, for any $x\in \sections{\galgebra}$:
 \begin{equation*}
 (d_* f+df)\circ x=-x\triangleright (d_* f)-\ld{x}^\galgebra (df)+\ld{df}^{\galgebra^*} x-x\vee (df)+d_*(\rho_{\galgebra}(x)f)+d(\rho_{\galgebra}(x)f).
 \end{equation*}
Among these terms,  it is easy to see that $-\ld{x}^\galgebra (df)+d(\rho_{\galgebra}(x)f)=0$.
Also, by (\ref{Eqt:anchor}) and (\ref{Eqt:another2}), we have
 $$\pairing{\ld{df}^{\galgebra^*} x}{\xi}=\pairing{x}{[df,\xi]}=-\pairing{x}{\phiUprotate(\xi)\triangleright df}=\pairing{\phiUprotate(\xi)\triangleright x}{df}=\rho(\ld{\xi}^{\galgebra^*} x)f=0.$$

Moreover, we find
 $$\pairing{-x\triangleright (d_* f)-x\vee (df)+d_*(\rho_{\galgebra}(x)f)}{\alpha}=\big(\rho_{\theta^*}(x\triangleright \alpha)-\rho_\galgebra(\alpha\triangleright x)-[\rho_\galgebra(x),\rho_{\theta^*}(\alpha)]\big)f=0.$$
 Thus we have proved $$(d_* f+df)\circ x=0.$$

 For $u\in \sections{\thetaalgebra}$, we have {
 $$\pairing{[d_* f,u]}{\alpha}=-\pairing{\phi(u)\triangleright d_* f}{\alpha}=\pairing{\phi(u)\triangleright \alpha}{d_* f}=\rho_{\theta^*}(\ld{u}^{\theta} \alpha)f=0,$$
 where we have used (\ref{Eqt:anchor}) and (\ref{Eqt:another2}). } Hence we obtain that $[d_* f,u]=0$, $\forall$ $u\in\sections{\thetaalgebra}$.   {Since $\galgebra^*$ and $\theta$ commute, we have $df\circ u=0$.}
 Consequently, we get
 $$(d_* f+df)\circ u=[d_* f,u]+df\circ u=0.$$

 Similarly, we can prove that $(d_* f+df)\circ (\alpha+\xi)=0$, for all  $\alpha\in \sections{{\theta^*}}$, $\xi\in\sections{{\galgebra^*}}$. This completes the proof of \eqref{CA:5}.

 So, to prove that $E$ is a Courant algebroid, it remains to show Condition \eqref{CA:1}, namely the Leibniz identity, for which we examine various cases.

\noindent$\bullet$
For three elements in $\galgebra\oplus\thetaalgebrastar$, the assumption that $(\galgebra,\thetaalgebrastar)$ is a matched pair already implies  Condition \eqref{CA:1} involving these elements.

\noindent$\bullet$
For  elements $x\in \sections{\galgebra},\xi,\eta\in \Gamma(\galgebrastar)$, by (\ref{Eqt:temp1}) and $\ld{[\xi,\eta]}=[\ld{\xi},\ld{\eta}]$, we have
\begin{eqnarray*}
&&x\circ (\xi\circ \eta)-(x\circ \xi)\circ \eta-\xi\circ (x\circ \eta)\\ &=&
\ld{x}[\xi,\eta]-\ld{[\xi,\eta]} x+x\vee [\xi,\eta]-\big([\ld{x}\xi,\eta]-\ld{\ld{\xi} x} \eta+\ld{\eta} \ld{\xi} x-(\ld{\xi} x) \vee \eta\big)\\ &&-\big([\xi,\ld{x}\eta]+\ld{\ld{\eta} x} \xi-\ld{\xi} \ld{\eta} x+(\ld{\eta} x) \vee \xi-(d_\galgebra+d_{\theta^*})\pairing{\ld{\eta} x}{\xi}\big)\\ &=&x\vee [\xi,\eta]+(\ld{\xi} x) \vee \eta-(\ld{\eta} x) \vee \xi+d_{\theta^*}\pairing{\ld{\eta} x}{\xi}.
\end{eqnarray*}
The last line vanishes because
\begin{eqnarray*}
&&\pairing{x\vee [\xi,\eta]+(\ld{\xi} x) \vee \eta-(\ld{\eta} x) \vee \xi}{\alpha}+\rho_{\theta^*}(\alpha)\pairing{\ld{\eta} x}{\xi}
\\ &=& {
\rho_{\theta^*}(\alpha)\pairing{x}{[\xi,\eta]}-\pairing{x}{\alpha\triangleright [\xi,\eta]}+\rho_{\theta^*}(\alpha)\pairing{\ld\xi x}{\eta}-\pairing{\ld\xi x}{\alpha\triangleright\eta}
}
\\ &&-{\rho_{\theta^*}(\alpha)\pairing{\ld\eta x}{\xi}+\pairing{\ld\eta x}{\alpha\triangleright\xi}-\rho_{\theta^*}(\alpha)\pairing{x}{[\eta,\xi]}
}
\\&=&\pairing{x}{-\alpha\moduleaction [\xi,\eta]+
[\alpha\moduleaction\xi,\eta]+[\xi,\alpha\moduleaction \eta]}=0,\qquad \forall \alpha\in\Gamma(\thetaalgebrastar).
\end{eqnarray*}
So we obtain
\begin{equation}
\label{666}x\circ (\xi\circ \eta)-(x\circ \xi)\circ \eta-\xi\circ (x\circ \eta)=0.
\end{equation}

\noindent$\bullet$
For elements $u\in \sections{\thetaalgebra},x\in\sections{\galgebra},\xi\in \Gamma(\galgebrastar)$, we have
\begin{eqnarray}\label{111}
u\circ(x\circ \xi)-(u\circ x)\circ \xi-x\circ(u\circ \xi)&=&
(\ld{\xi}x)\triangleright u+[u,x\vee \xi].
\end{eqnarray}
We need the following identity $$\phi(x\vee \xi)=\ld{\xi} x.$$ In fact, for any $\eta\in \Gamma(\galgebra^*)$, we have
 {(note that $\rho\circ \phiUprotate=0$)}
\begin{eqnarray*}
\pairing{\phi(x\vee \xi)}{\eta}&=&-\pairing{(x\vee \xi)}{\phiUprotate(\eta)}=\pairing{x}{\phiUprotate(\eta)\triangleright\xi}\\ &=&\pairing{x}{[\eta,\xi]}=\pairing{\ld{\xi} x}{\eta}.
\end{eqnarray*}

Then,  by the fact that $[u,x\vee \xi]=-\phi(x\vee \xi)\triangleright u$, the RHS of \eqref{111} is indeed zero.
So we have proved
\begin{eqnarray}
\label{777}u\circ(x\circ \xi)-(u\circ x)\circ \xi-x\circ(u\circ \xi)=0.
\end{eqnarray}
In a similar manner, one can show that
\begin{eqnarray}
\label{333} (\ld{u} \alpha)\moduleaction\xi+[\xi,\alpha\vee u]=0,\qquad \forall u\in \sections{\thetaalgebra},\alpha\in \Gamma(\thetaalgebrastar),\xi\in\Gamma(\galgebrastar).
\end{eqnarray}

\noindent$\bullet$
For elements $u\in \sections{\thetaalgebra},x\in\sections{\galgebra},\alpha\in \Gamma(\thetaalgebrastar)$, we have
\begin{eqnarray*}
&&x\circ(\alpha\circ u)-(x\circ \alpha)\circ u-\alpha\circ (x\circ u)\\&=&x\circ (\ld{\alpha} u-\ld{u} \alpha+\alpha\vee u)-(x\moduleaction \alpha-\alpha\moduleaction x)\circ u-\big(\ld{\alpha} (x\triangleright u)-\ld{x\triangleright u}\alpha+\alpha\vee (x\triangleright u)\big).
\end{eqnarray*}
According to Proposition \ref{Prop:D'sdoublebracket}, we can show that the $\sections{\galgebra}$-component of $x\circ(\alpha\circ u)-(x\circ \alpha)\circ u-\alpha\circ (x\circ u)$ is zero. In fact, for any $\xi\in\Gamma(\galgebrastar)$,
\begin{eqnarray*}
&&\pairing{\mathrm{Pr}_{\sections{\galgebra}}\bigl(x\circ(\alpha\circ u)-(x\circ \alpha)\circ u-\alpha\circ (x\circ u) \bigl)}{\xi}\\&=&\pairing{(\ld{u} \alpha)\moduleaction x-\ld{\alpha\vee u} x}{\xi}\\&=&\pairing{x}{-(\ld{u} \alpha)\moduleaction\xi+[\alpha\vee u,\xi]}=0,
\end{eqnarray*}
where we have used \eqref{333}.

To evaluate the $\sections{\thetaalgebra}$-component, we take  $\beta\in\Gamma(\thetaalgebrastar)$, and compute
\begin{eqnarray*}
&&\pairing{\mathrm{Pr}_{\sections{\thetaalgebra}}\bigl(x\circ(\alpha\circ u)-(x\circ \alpha)\circ u-\alpha\circ (x\circ u) \bigl)}{\beta}\\&=&\pairing{x\moduleaction (\ld{\alpha} u)+x\vee (\alpha\vee u)-\ld{x\moduleaction \alpha} u+(\alpha\moduleaction x)\moduleaction u-\ld{\alpha}(x\moduleaction u)}{\beta}\\&=&\rho_{\galgebra}(x)\rho_{\theta^*}(\alpha)\beta(u)-\rho_{\galgebra}(x)\pairing{u}{[\alpha,\beta]}-
\rho_{\theta^*}(\alpha)\pairing{u}{x\triangleright \beta}+\pairing{u}{[\alpha,x\triangleright \beta]}\\ &&
+\pairing{\alpha\vee u}{\beta\triangleright x}-\rho_{\theta^*}(x\triangleright \alpha)\beta(u)+\pairing{u}{[x\triangleright \alpha,\beta]}+\rho_{\galgebra}(\alpha\triangleright x)\beta(u)-\pairing{u}{(\alpha\triangleright x)\triangleright \beta)}\\ &&-\rho_{\theta^*}(\alpha)\rho_{\galgebra}(x)\beta(u)+
\rho_{\theta^*}(\alpha)\pairing{u}{x\triangleright \beta}+\rho_{\galgebra}(x)\pairing{u}{[\alpha,\beta]}-\pairing{u}{x\triangleright [\alpha,\beta]}\\ &=&0.
\end{eqnarray*}
Here we have used the facts that
$$\pairing{\alpha\vee u}{\beta\triangleright x}=\pairing{(\beta\triangleright x)\triangleright \alpha}{u},$$
and $(\galgebra,\theta^*)$ is a matched pair.

For the $\Gamma(\thetaalgebrastar)$-component, we have, for any $v\in \sections{\thetaalgebra}$,
\begin{eqnarray*}
&&\pairing{\mathrm{Pr}_{\Gamma(\thetaalgebrastar)}(x\circ(\alpha\circ u)-(x\circ \alpha)\circ u-\alpha\circ (x\circ u))}
{v}\\&=&\pairing{-x\moduleaction(\ld{u}\alpha)+\ld{u}(x\moduleaction \alpha)+\ld{x\moduleaction u} \alpha}{v}\\&=&-\rho_{\galgebra}(x)\pairing{\ld{u} \alpha}{v}+\pairing{\ld{u} \alpha}{x\triangleright v}-\pairing{x\moduleaction \alpha}{[u,v]}-\pairing{\alpha}{[x\moduleaction u,v]}\\&=&\pairing{\alpha}{x\moduleaction [u,v]-[u,x\moduleaction v]-[x\moduleaction u,v]}=0.
\end{eqnarray*}

For the $\Gamma(\galgebrastar)$-component, we have, for any $y\in \sections{\galgebra}$,
\begin{eqnarray*}
&&\pairing{\mathrm{Pr}_{\Gamma(\galgebrastar)}(x\circ(\alpha\circ u)-(x\circ \alpha)\circ u-\alpha\circ (x\circ u))}
{y}\\&=&\pairing{\ld{x} (\alpha\vee u)-(x\moduleaction \alpha)\vee u-\alpha\vee (x\moduleaction u)}{y}\\&=&\rho_{\galgebra}(x)\rho_{\galgebra}(y)\alpha(u)-\rho_{\galgebra}(x)\pairing{\alpha}{y\triangleright u}-\pairing{\alpha\vee u}{[x,y]}\\ &&-\rho_{\galgebra}(y)\pairing{x\moduleaction \alpha}{u}+\pairing{x\moduleaction \alpha}{y\moduleaction u}-\rho_{\galgebra}(y)\pairing{x\triangleright u}{\alpha}+\pairing{\alpha}{y\triangleright (x\triangleright u)}\\&=&\pairing{\alpha}{-x\moduleaction(y\moduleaction u)+y\moduleaction (x\moduleaction u)+[x,y]\moduleaction u}=0.
\end{eqnarray*}
Combining  these four facts, we have proved that
\begin{eqnarray}
\label{444} x\circ(\alpha\circ u)-(x\circ \alpha)\circ u-\alpha\circ (x\circ u)=0.
\end{eqnarray}

\noindent$\bullet$
For elements $x,y\in \sections{\galgebra},\xi\in\Gamma(\galgebrastar)$, we have
\begin{eqnarray*}
&&x\circ(y\circ \xi)-(x\circ y)\circ \xi-y\circ (x\circ \xi)\\&=&
x\circ (\ld{y} \xi-\ld{\xi} y+y\vee \xi)-(\ld{[x,y]} \xi-\ld{\xi} [x,y]+[x,y]\vee \xi)-y\circ (\ld{x} \xi-\ld{\xi} x+x\vee \xi)
\\ &=&(\ld{x} \ld{y} \xi-\ld{[x,y]} \xi-\ld{y} \ld{x} \xi)+\big(-\ld{\ld{y}\xi}x-[x,\ld{\xi} y]+\ld{\xi}[x,y]+\ld{\ld{x}\xi}y-[\ld{\xi} x,y]\big)\\ &&+\big(x\vee (\ld{y} \xi)+x\triangleright(y\vee \xi)-[x,y]\vee \xi-y\vee (\ld{x} \xi)-y\triangleright(x\vee \xi)\big).
\end{eqnarray*}
Obviously, the first two terms (the $\Gamma(\galgebra^*), \sections{\galgebra}$ components) vanish due to $[\ld{x},\ld{y}]=\ld{[x,y]}$ and Eqt. \eqref{555}.
We show that the last term, namely the $\sections{\thetaalgebra}$-component, vanishes. In fact, we note that, for any $\alpha\in\Gamma(\thetaalgebra^*)$,
\begin{eqnarray*}
\pairing{x\vee (\ld{y} \xi)}{\alpha}&=&\pairing{\alpha\triangleright x}{\ld{y} \xi} =\rho_{\galgebra}(y)\pairing{\xi}{\alpha\triangleright x}-\pairing{\xi}{[y,\alpha\triangleright x]}.
\end{eqnarray*}
 {By \eqref{Eqt:Vee1}, we have}
\begin{eqnarray*}
\pairing{x\triangleright (y\vee \xi)}{\alpha}&=&\rho_{\galgebra}(x)\pairing{\alpha}{y\vee \xi}-\pairing{x\triangleright \alpha}{y\vee \xi}\\&=&\rho_{\galgebra}(x)\pairing{\xi}{\alpha\triangleright y}-\pairing{\xi}{(x\triangleright \alpha)\triangleright y}.
\end{eqnarray*}

It follows that
\begin{eqnarray*}
&&\pairing{x\vee (\ld{y} \xi)+x\triangleright(y\vee \xi)-[x,y]\vee \xi-y\vee (\ld{x} \xi)-y\triangleright(x\vee \xi)}{\alpha}\\&=&\rho_{\galgebra}(y)\pairing{\xi}{\alpha\triangleright x}-\pairing{\xi}{[y,\alpha\triangleright x]}
+\rho_{\galgebra}(x)\pairing{\xi}{\alpha\triangleright y}-\pairing{\xi}{(x\triangleright \alpha)\triangleright y}\\ &&-\pairing{\xi}{\alpha\triangleright [x,y]}-\rho_{\galgebra}(x)\pairing{\xi}{\alpha\triangleright y}+\pairing{\xi}{[x,\alpha\triangleright y]}
-\rho_{\galgebra}(y)\pairing{\xi}{\alpha\triangleright x}+\pairing{\xi}{(y\triangleright \alpha)\triangleright x}\\ &=&
0,
\end{eqnarray*}
where we have used the fact that $(\galgebra,\theta^*)$ is a matched pair.
We thus proved
\begin{equation}
\label{888} x\circ(y\circ \xi)-(x\circ y)\circ \xi-y\circ (x\circ \xi)=0.
\end{equation}

Now  we have proved the Leibniz identity of the forms \eqref{666}, \eqref{777}, \eqref{444} and \eqref{888}.
Using Conditions \eqref{CA:4}--\eqref{CA:5}, it is easy to see that the Leibniz identity  also holds for different orders of elements involved in the above four cases.
Since the role of $\galgebra$ and $\thetaalgebrastar$ are symmetric, the Leibniz identity actually holds for all possible  choices of elements involved.

In summary, we have proved that $E=A_{\galgebra\triangleright\theta}\oplus A_{\theta^*\triangleright\galgebra^*}$ is a Courant algebroid, and thus $A_{\galgebra\triangleright\theta}$ and $A_{\theta^*\triangleright \galgebra^*}$ compose  a Lie bialgebroid.
This completes the proof of Theorem \ref{Thm:mathchedpairiff}.

\end{proof}

\begin{prop}\label{Prop:invariancy}
 Given two vector bundles
$\thetaalgebra$ and $\galgebra$ over $M$, and a bundle map
$\phi:~\thetaalgebra\lon\galgebra$, we define a pairing
$$
\bookpairing{\tobefilledin,\tobefilledin}:~\Gamma(\galgebrastar) \times \sections{\thetaalgebra} \lon C^\infty(M)
$$by setting
$$
\bookpairing{\xi,u}\defbe
\pairing{\xi}{\phi(u)},\qquad\forall~\xi\in\Gamma(\galgebrastar),u\in
\sections{\thetaalgebra}.
$$

\begin{enumerate}
\item[1)] If  $\galgebra$ is a Lie algebroid and $\thetaalgebra$ is a
$\galgebra$-module, then there is a Lie algebroid crossed module
structure   $\crossedmoduletriple{\thetaalgebra}{\phi}{\galgebra} $
if and only if the following two equalities hold:
\begin{eqnarray}\label{Eqn:t1}
\rho_{\galgebra}(x)\bookpairing{\xi,u}&=&
\bookpairing{\ld{x}^{\galgebra}\xi,u}+\bookpairing{\xi,x\moduleaction u} ,\\
\label{Eqn:t3}
    \bookpairing{\alpha\vee v,u}&=&-\bookpairing{\alpha\vee u,v}
    ,
\end{eqnarray}
for all $x\in
\sections{\galgebra},\alpha\in\Gamma(\thetaalgebrastar),\xi\in\Gamma(\galgebrastar),u,v\in\sections{\thetaalgebra}$.
(The $ \vee $ operator is defined by Eqt. \eqref{Eqt:Vee2}.)

\item[2)] If  $\thetaalgebrastar$ is a Lie algebroid  and
$\galgebrastar$ is a $\thetaalgebrastar$-module, then there is a Lie
algebroid crossed module structure  $\crossedmoduletriple{\galgebrastar}{\phiUprotate
}{\thetaalgebrastar}$ if and only if the following two equalities
hold:
 \begin{eqnarray*}
\rho_{\theta^*}(\alpha)\bookpairing{
\xi,u}&=&\bookpairing{\alpha\moduleaction
\xi,u}+\bookpairing{\xi,\ld{\alpha}^{\theta^*} u},\\
    \bookpairing{\xi,x\vee \eta}&=&-\bookpairing{\eta,x\vee \xi}
    ,
 \end{eqnarray*}
  for all
$x\in
\sections{\galgebra},\alpha\in\Gamma(\thetaalgebrastar),\xi,\eta\in\Gamma(\galgebrastar),u
\in\sections{\thetaalgebra}$. (The $ \vee $ operator is defined by Eqt. \eqref{Eqt:Vee1}.)

\end{enumerate}
\end{prop}
\begin{proof} We first show that Eqt.  \eqref{Eqn:t1} is equivalent to
\begin{equation*}
\phi ( x\moduleaction u)=\ba{x}{\phi
(u)}.
\end{equation*}
In fact, by definition, we have
\begin{align*}&\rho_{\galgebra}(x)\bookpairing{\xi,u}-\bookpairing{\ld{x}^{\galgebra}\xi,u}-\bookpairing{\xi,x\moduleaction u}\\=&\rho_{\galgebra}(x)\pairing{\xi}{\phi(u)}-\rho_{\galgebra}(x)\pairing{ \xi}{\phi(u)}+\pairing{\xi}{[x,\phi(u)]}-\pairing{\xi}{\phi( x\moduleaction u)}
\\=&\pairing{\xi}{[x,\phi (u)]-\phi ( x\moduleaction u)}.
\end{align*}
It can be easily seen  that Eqt.  \eqref{Eqn:t3} is equivalent to
\begin{equation*}
 {\phi
(u)} \moduleaction v=-{\phi (v)} \moduleaction u ,\quad\forall~
u,v\in \sections{\thetaalgebra}.
\end{equation*}
We thus get the first conclusion, by Lemma \ref{Lem:ensentialcrossedmoduleLiealgebra}. The second is proved analogously.
\end{proof}

\begin{prop}\label{Prop:PQgivesLieBialgebraCrossedModule}
Given a matched pair $(P,Q)$ and suppose that there is a
bilinear pairing
$$
\bookpairing{\tobefilledin,\tobefilledin}:~\Gamma(\Ps) \times \Gamma(\Qs) \lon C^\infty(M)
$$
satisfying the following conditions: for all $x\in \sections{P},\alpha\in
\sections{Q},\xi,\eta\in \Gamma(\Ps),u,v\in\Gamma(\Qs)$,
\begin{align*}
\rho_P(x)\bookpairing{\xi,u}&=\bookpairing{\ld{x}^P\xi,u}+\bookpairing{\xi,x\moduleaction
u},\\
\rho_Q(\alpha)\bookpairing{\xi,u}&=
\bookpairing{\alpha\moduleaction
\xi,u}+\bookpairing{\xi,\ld{\alpha}^Q u},\\
 \bookpairing{\alpha\vee u,v}&=-\bookpairing{\alpha\vee
v ,u},\\
\bookpairing{x\vee \eta,\xi}&=-\bookpairing{x\vee \xi,\eta}.
\end{align*}

Then there associates a Lie bialgebroid crossed module structure on
$\crossedmoduletriple{\Qs }{\phi}{P}$ and
$\crossedmoduletriple{\Ps }{\phi^T}{Q}$ , where $\phi:~\Qs\lon P$ is
determined by \begin{equation*}
\pairing{\xi}{\phi(u)}\defbe
\bookpairing{\xi,u},\quad\forall~  \xi\in \Gamma(\Ps),u\in\Gamma(\Qs).
\end{equation*}
\end{prop}
\begin{proof}Let us denote $\Qs=\thetaalgebra$ and $P=\galgebra$.
Hence $\galgebra$ acts on $\thetaalgebra$ and $\thetaalgebrastar$
acts on $\galgebra^*$. The four conditions exactly match those in Proposition \ref{Prop:invariancy}. Then by Theorem
\ref{Thm:mathchedpairiff}, the conclusion is clear.
\end{proof}

\begin{proof}[Proof of Proposition \ref{Prop:manin3}] We   show that  the $L$-invariance condition Eqt.  \eqref{Eqt:bookpairinginvariant}
is equivalent to  the four conditions in Proposition \ref{Prop:PQgivesLieBialgebraCrossedModule}. Then the
conclusion follows immediately.

Under the decompositions
$$
K=P\oplus Q,\qquad   K^*= Q^0\oplus P^0,
$$
we can prove the following two identities
\begin{eqnarray*}  \ld{x}^{K}(\xi+u)
&=&\ld{x}^{P}\xi+(x\vee\xi+x\moduleaction u),\\
\ld{\alpha}^{K}(\xi+u) &=& (\alpha \moduleaction
\xi+\alpha\vee u)+\ld{\alpha}^{Q}u
\end{eqnarray*}
for all $x\in \sections{P},\alpha\in \sections{Q}$, $\xi\in \Gamma(Q^0),u\in  \Gamma(P^0)$.  {Here $x\vee\xi\in \sections{P^0}$ is defined by $$\pairing{x\vee\xi}{\beta}=\pairing{\xi}{\beta\triangleright x}$$ and $x\triangleright u$ comes from the dual action of $P$ on $Q$.}
  {In fact,  for any $y\in \sections{P}$, $\beta\in \sections{Q}$,
\begin{eqnarray*}
\pairing{\ld{x}^{K}\xi}{y+\beta}=\rho_{P}(x)\xi(y)-\pairing{\xi}{[x,y]+x\triangleright \beta-\beta\triangleright x}
=\pairing{\ld{x}^{P}\xi+x\vee\xi}{y+\beta}.
\end{eqnarray*}
and
$$ \pairing{\ld{x}^K u}{y+\beta}=\rho_{P}(x)\beta(u)-\pairing{u}{x\triangleright \beta}=\pairing{x\triangleright u}{\beta}.$$
}
Thus the equivalence between Eqt.
\eqref{Eqt:bookpairinginvariant} and  the four equalities in Proposition \ref{Prop:PQgivesLieBialgebraCrossedModule} is obvious.
\end{proof}

Now, we can complete the proof of Proposition \ref{Prop:manin3 reverse}.

\begin{proof}[Proof of Proposition \ref{Prop:manin3 reverse}]
 {By Theorem \ref{Thm:mathchedpairiff}, we have that $(\galgebra,\theta^*)$ is a matched pair, so  $K=\galgebra\bowtie \thetaalgebrastar$ is a Lie algebroid. Moreover, by Proposition \ref{Prop:invariancy} and
the the proof of Proposition
\ref{Prop:manin3}, we get the $K$-invariance of $\bookpairing{\tobefilledin,\tobefilledin}$. The fact that  $\galgebra$ and $\theta^*$ are transverse Dirac structures of $K$ is obvious.
}
\end{proof}




\end{document}